\documentclass[onefignum,onetabnum]{siamart171218}

\usepackage[T1]{fontenc}
\usepackage[utf8]{inputenc}
\usepackage{CJK}
\usepackage{amsfonts,amssymb,bm,mathrsfs,indentfirst}
\usepackage{amsmath}
\usepackage{verbatim}
\usepackage{color}
\usepackage{stmaryrd}
\usepackage{graphicx}
\newtheorem*{remark}{Remark}
\newtheorem*{claim}{Claim}
\allowdisplaybreaks

\title{The Maximum Principle for progressive optimal stochastic control problems with random jumps\thanks{Submitted to the editors
DATE.
\funding{This work was supported by the Natural Science Foundation of China (11631104, 61573217, 11831010),
the National High-level personnel of special support program, and the Chang Jiang Scholar Program of Chinese
Education Ministry.}}}
\author{Yuanzhuo Song\thanks{ School of Mathematics and Zhongtai Securities Institute for Financial Study, Shandong University, Jinan 250100, China
(\email{201511380@mail.sdu.edu.cn}).}
\and Shanjian Tang \thanks{School of Mathematical Sciences, Fudan University, Shanghai 200433, China (\email{sjtang@fudan.edu.cn}).}
\and Zhen Wu\thanks{Corresponding author. School of Mathematics and Zhongtai Securities Institute for Financial Study, Shandong University, Jinan 250100, China (\email{wuzhen@sdu.edu.cn}
).}
}
\headers{Maximum Principle with random jumps}{Yuanzhuo Song, Shanjian Tang,  Zhen Wu}

\date{}
\begin{document}

\maketitle
\begin{abstract}
  In this paper, we obtain the maximum principle for optimal controls of stochastic systems with jumps by introducing a new method of variation. The control is allowed to enter both diffusion and jump term and the control domain need not to be convex.
\end{abstract}



\section{Introduction}
Stochastic optimal control problem is an important problem in control theory. Maximum principle, the necessary condition for the optimal control, is one of the central results. A lot of work has been done on this topic, Peng \cite{peng1990general} proved the general maximum principle for forward stochastic control system without jump by using second-order variation equation to overcome the difficulty appeared along with the non-convex control domain and control entering the diffusion term. Situ \cite{situ1991maximum} obtained the maximum principle for forward stochastic control system with jumps, but in his system the jump coefficient doesn't contain the control variable. Tang and Li \cite{Tang1994necessary} proved the maximum principle for forward control system where the control variable is allowed into both diffusion and jump coefficients. There are many results for other stochastic control systems, we refer the reader for Peng \cite{peng1993backward}, Wu \cite{zhen1998maximum}, Shi and Wu \cite{shi2006maximum} for forward-backward system.

 In this paper, we consider optimal control of progressive stochastic differential systems with random jumps, where the integrand of stochastic integrals w.r.t. the compensated Poisson point process  could be progressively measurable instead of predictable as in Tang and Li \cite{Tang1994necessary}. In this new setting, the incorrect estimate (i.e. the third one) in (2.10) of \cite{Tang1994necessary} is not required anymore by considering only those perturbed admissible controls which admit no perturbation to the optimal control at the jumping times of the underlying point process.


The rest of this paper is organized as follows. In section 2, we give some preliminaries about the stochastic integral with respect to jumps. The difference between our model with the model in \cite{Tang1994necessary} is that we need the integrand to be progressive in order to make our variation effective. Our main results are stated in Sections 3, 4 and 5. In these sections, we employ the new spike variation and introduce second order variation equations to get the desired maximum principle, which is the rigorous version in strict mathematical framework. In section 6, we explain the characteristic of our results and show our future research directions. Some results about stochastic differential equation (SDE) with jumps are put in appendix.

\section{Preliminaries}
\label{sec:main}

Let $(\Omega, \mathscr{F}, \{\mathscr{F}_t\}_{t\ge 0}, P)$ be a complete probability space with filtration, and on the probability space,  there is
    a $\mathscr{F}_t$-Brownian motion $\{B_t\}_{t\ge 0}$; and
    a Poisson random measure $N$ on $R_+\times E$ adapted to $\mathscr{F}_t$, where $E$ is a standard measure space with a $\sigma$-field $\mathscr{E}$. The mean measure of $N$ is a measure on $(R_+\times E, \mathscr{B}(R_+)\otimes\mathscr{E})$ which has the form $Leb\times \lambda$, where $Leb$ denotes the Lebesgue measure on $R_+$ and $\lambda$ is a finite measure on $E$. For any $B\in \mathscr{E}$ and $t\in R_+$, since $\lambda(B)<\infty$, we set $\tilde{N}(\omega, [0, t]\times B):=N(\omega, [0, t]\times B)-t\lambda(B)$. It is well known that $\tilde{N}(\omega, [0, t]\times B)$ is a martingale for every $B$.
    We assume that $\{\mathscr{F}_t\}_{t\ge 0}$ is generated by $B, N$, that is
    \[\mathscr{F}_t:=\sigma\left(N([0, s], A), 0\le s\le t, A\in\mathscr{E}\right)\vee\sigma(B_s, 0\le s\le t)\vee\mathscr{N}\]
    where $\mathscr{N}$ denotes the totality of $P$-null sets. Then $\mathscr{F}_t$ satisfies the usual condition.

    Suppose that $M$ is a Euclid space, $\mathscr{B}(M)$ is the Borel $\sigma$-field on $M$. Given $T>0$, a process $X:[0, T]\times\Omega\rightarrow M$
    is called progressive (predictable) if $X$ is $\mathscr{G}/\mathscr{B}(M)$ ($\mathscr{P}/\mathscr{B}(M)$) measurable, where $\mathscr{G}$ ($\mathscr{P}$) is the progressive (predictable) $\sigma$-field on $[0, T]\times\Omega$; a process $X:[0, T]\times\Omega\times E\rightarrow M$ is called $E$-progressive ($E$-predictable) if $X$ is $\mathscr{G}\otimes\mathscr{E}/\mathscr{B}(M)$($\mathscr{P}\otimes\mathscr{E}/\mathscr{B}(M)$) measurable. On contrast to \cite{Tang1994necessary}, the stochastic integral we used is more general, that is the integrand of the stochastic integral in our paper is $E$-progressive rather than $E$-predictable.

    Now we introduce some notations. Given a process $X_t$ with c\`adl\`ag paths, $X_{0-}:=0$ and $\Delta X_t:=X_t-X_{t-},t\ge 0$. Let $\mu$ denote the measure on $\mathscr{F}\otimes\mathscr{B}([0, T])\otimes\mathscr{E}$ generated by $N$ that $\mu(A)=E\int_0^T\int_E I_A N(ds, de)$. For any $\mathscr{F}\otimes\mathscr{B}([0, T])\otimes\mathscr{E}/\mathscr{B}(R)$ measurable integrable process $X$, we set $\mathbb{E}[X]:=\int X d\mu$ and denote by $\mathbb{E}[X|\mathscr{P}\otimes\mathscr{E}]$  the Radon-Nikodym derivatives with respect to $\mathscr{P}\otimes\mathscr{E}$. Note that $\mathbb{E}$ is not an expectation (for $\mu$ is not a probability measure), though it has similar properties to expectation. Then we introduce the definition of stochastic integral of random measure which is more general than that in \cite{Tang1994necessary} based on the theory of stochastic integral of process. We will use the theory of dual predictable projection (also called compensator) and we will not give the definition here. The definition and other details of the theory can be found in \cite{he2018semimartingale}.

    Suppose $H=I_{A\times B}, A\in \mathscr{G}, B\in\mathscr{E}$. We define
    \[\int_0^T\int_E H \tilde{N}(dt, de):=\int_0^T I_A\tilde{N}(dt, B)\]
    Then for any $E$-progressive simple function with the form $H=\sum_{i=1}^n a_iI_{A_i\times B_i},a_i\in R,A_i\in \mathscr{G}, B_i\in\mathscr{E}$, we can define by linear extension.

    For $E$-progressive process $H$ that $E\left[\int_0^T\int_E H^2N(dt, de)\right]<\infty$, there exist a sequence of $E$-progressive simple functions $H_n$ which have the form above that $$\lim_{n\rightarrow \infty}E\left[\int_0^T\int_E (H-H_n)^2N(dt, de)\right]=0$$. \\We can verify that $\{(H_n. \tilde{N})_T\}_{n\ge 1}$ is a cauchy sequence in $L^2$, so we define
    \[\int_0^T\int_E H \tilde{N}(dt, de):=(L^2)\lim_{n\rightarrow \infty}\int_0^T \int_EH_n\tilde{N}(dt, de)\]
    \begin{proposition}
        If $H$ is a positive $E$-progressive process that \[E\left[\int_0^T\int_E HN(dt, de)\right]<\infty,\] then
        \begin{equation}\label{ss}
        \left(\int_0^{\cdot}\int_EHN(ds, de)\right)_t^p=\int_0^t\int_E\mathbb{E}\left[H|\mathscr{P}\otimes\mathscr{E}\right]\lambda(de)ds
        \end{equation}
        where $X^p$ is the dual predictable projection of $X$.
    \end{proposition}

    \begin{proof}
    If $H=I_{A\times B}, A\in \mathscr{G}, B\in\mathscr{E}$, then
    \[\left(\int_0^{\cdot}\int_EHN(ds, de)\right)_t^p=\left(\int_0^{\cdot} I_AN(ds, B)\right)_t^p=\int_0^t \mathbb{E}_B\left[I_A|\mathscr{P}\right]\lambda(B)ds\]
    where $\mathbb{E}_B$ is the measure on $\mathscr{B}([0, T])\otimes\mathscr{F}$ generated by $N([0, t]\times B)$. Now we need a claim.
    \begin{claim}
    $$\int_E\mathbb{E}\left[I_{A\times B}|\mathscr{P}\otimes\mathscr{E}\right]\lambda(de)=\lambda (B)\mathbb{E}_B\left[I_A|\mathscr{P}\right]$$
    \end{claim}
    \begin{proof}
    It is obvious that both sides of the equation are predictable. Now for any $C\in \mathscr{P}$,
    \begin{equation*}
    \begin{aligned}
    &E\left[\int_0^TI_C\int_E\mathbb{E}\left[I_{A\times B}|\mathscr{P}\otimes\mathscr{E}\right]\lambda(de)dt\right]=E\left[\int_0^T\int_E\mathbb{E}\left[I_CI_{A\times B}|\mathscr{P}\otimes\mathscr{E}\right]N(dt, de)\right]\\
    &=E\left[\int_0^T\int_EI_{A\cap C}I_BN(dt, de)\right]=E\left[\int_0^T\int_EI_{A\cap C}N(dt, B)\right].
    \end{aligned}
    \end{equation*}
    On the other hand,
    \begin{equation*}
    \begin{aligned}
    &E\left[\int_0^TI_C\lambda (B)\mathbb{E}_B\left[I_A|\mathscr{P}\right]dt\right]=E\left[\int_0^T\int_EI_B\mathbb{E}_B\left[I_{A\cap C}|\mathscr{P}\right]\lambda(de)dt\right]\\
    &=E\left[\int_0^T\int_EI_B\mathbb{E}_B\left[I_{A\cap C}|\mathscr{P}\right]N(dt, de)\right]=E\left[\int_0^T\int_EI_{A\cap C}N(dt, B)\right].
    \end{aligned}
    \end{equation*}
    The proof is then complete.
    \end{proof}

    We come back to the proof of proposition. The claim above shows that \cref{ss} is true for functions of the form $I_{A\times B}$, now we define
    \[\mathscr{C}=\{H=I_{A\times B}\mid A\in\mathscr{G}, B\in\mathscr{E}\}\]
    $\mathscr{C}$ is a $\pi$-system that generate $\mathscr{G}\times\mathscr{E}$. Define
    \begin{equation*}
    \begin{aligned}
    \mathscr{H}&=\Bigg\{H\text{ is bounded and $E$-progressive}\mid \\ & \left(\int_0^{\cdot}\int_EHN(ds, de)\right)_t^p=\int_0^t\int_E\mathbb{E}\left[H|\mathscr{P}\otimes\mathscr{E}\right]\lambda(de)ds\Bigg\}.
    \end{aligned}
    \end{equation*}
    Then $\mathscr{C}\in\mathscr{H}$, and by the linear property of dual predictable projection we can verify that $\mathscr{H}$ is a linear space. If $H^n\uparrow H$, and $H$ is bounded, then we have $\left(\int_0^{\cdot}\int_EH^nN(ds, de)\right)_t^p\rightarrow \left(\int_0^{\cdot}\int_EHN(ds, de)\right)_t^p$ for each $t$ in $L^1$ sense and this implies that $H\in\mathscr{H}$. So, by monotone class theorem, we prove that all bounded $E$-progressive process satisfy the result.

    For $E$-progressive $H$ such that $E\left[\int_0^T\int_E HN(dt, de)\right]<\infty$, we set
    $$H^n=HI_{\{|H|\le n\}}\in \mathscr{H}$$
    and take limit to show that $H$ satisfies \cref{ss}.
    \end{proof}

    \begin{proposition}
    If $H$ is $E$-progressive and satisfies
    $E\left[\int_0^T\int_E H^2 N(dt, de)\right]<\infty, $
    then we have
    \begin{equation*}
    \int_0^T\int_E H\tilde{N}(dt, de)=\int_0^T\int_EHN(dt, de)-\left(\int_0^{\cdot}\int_EHN(dt, de)\right)^p_T.
    \end{equation*}
    \end{proposition}

    \begin{proof}
    If $H=I_{A\times B}, A\in\mathscr{G}, B\in\mathscr{E}$, by the definition of stochastic integral,
    \begin{equation*}
    \begin{aligned}
    \int_0^T\int_EH\tilde{N}(dt, de)&=\int_0^T I_A\tilde{N}(dt, B)=\int_0^T I_AN(dt, B)-\left(\int_0^{\cdot} I_AN(dt, B)\right)^p_T
    \\&=\int_0^T\int_EHN(dt, de)-\left(\int_0^{\cdot}\int_EHN(dt, de)\right)^p_T
    \end{aligned}
    \end{equation*}
    So for any $E$-progressive simple process which has the form $H=\sum_{i=1}^n a_iI_{A_i\times B_i}$ with $a_i\in R,A_i\in \mathscr{G}$, and $B_i\in\mathscr{E}$,  the conclusion is true.

    If $H$ is positive and $E\left[\int_0^T\int_E H^2 N(dt, de)\right]<\infty$, then $E\left[\int_0^T\int_E H N(dt, de)\right]<\infty$ and there exists a sequence of positive increasing simple functions $H_n$ with the above form  such that $$E\left[\int_0^T\int_E (H-H_n)^2 N(dt, de)\right]\rightarrow 0$$ as $n$ goes to infinity, so
    \begin{equation*}
    \begin{aligned}
    &\int_0^T\int_EH\tilde{N}(dt, de)=(L^2)\lim_{n\rightarrow \infty}\int_0^T\int_EH_n\tilde{N}(dt, de)\\&=(L^2)\lim_{n\rightarrow \infty}\left(\int_0^T\int_EH_nN(dt, de)-\left(\int_0^{\cdot}\int_EH_nN(dt, de)\right)^p_T\right)\\
    &=\int_0^T\int_EHN(dt, de)-(L^1 \text{ or } L^2)\lim_{n\rightarrow \infty}\left(\int_0^{\cdot}\int_EH_nN(dt, de)\right)^p_T\\
    &=\int_0^T\int_EHN(dt, de)-\left(\int_0^{\cdot}\int_EHN(dt, de)\right)^p_T.
    \end{aligned}
    \end{equation*}
    If $H$ is not positive, we decompose $H=H^+-H^-$ and get the result.
    \end{proof}

    From the last two propositions, we have
    \begin{proposition}
    If $H$ is $E$-progressive and $E\left[\int_0^T\int_E H^2N(dt, de)\right]<\infty$, then
    \[\int_0^T\int_EH\tilde{N}(dt, de)=\int_0^T\int_EHN(dt, de)-\int_0^T\int_E\mathbb{E}\left[H|\mathscr{P}\otimes\mathscr{E}\right]\lambda(de)dt\].
    \end{proposition}

    \begin{remark}
    Under the condition of the proposition above, we have
    \begin{equation*}
    \begin{aligned}
    E\left[\int_0^T\int_EHN(dt, de)\right]=E\left[\int_0^T\int_E\mathbb{E}\left[H|\mathscr{P}\otimes\mathscr{E}\right]\lambda(de)dt\right].
    \end{aligned}
    \end{equation*}
    In particular, if $H$ is $E$-predictable, we have the well-known result
    \begin{equation*}
    \begin{aligned}
    E\left[\int_0^T\int_EHN(dt, de)\right]=E\left[\int_0^T\int_EH\lambda(de)dt\right].
    \end{aligned}
    \end{equation*}
    \end{remark}

    Since $N([0, t]\times A)$ is quasi-left-continuous for each $A\in\mathscr{E}$, we have the following two propositions.
    \begin{proposition}
    If $H$ is $E$-progressive and $E\left[\int_0^T\int_E H^2 N(dt, de)\right]<\infty$, then we have
    \[\Delta(H. \tilde{N})_t=\int_E HN(\{t\}, de)\].
    \end{proposition}

    \begin{proof}
    If $H=I_{A\times B}, A\in\mathscr{G}, B\in\mathscr{E}$, then
    \begin{equation*}
    \begin{aligned}
    \Delta \left(\int_0^{\cdot}\int_EH\tilde{N}(ds, de)\right)_t\!=\!\Delta \left(\int_0^{\cdot}I_A\tilde{N}(ds, B)\right)_t\!=\!I_AN(\{t\}, B)\!=\!\int_EI_{A\times B}N(\{t\}, de)
    \end{aligned}
    \end{equation*}

    So for any $E$-progressive simple functions with the form $H=\sum_{i=1}^n a_iI_{A_i\times B_i},a_i\in R,A_i\in \mathscr{G}, B_i\in\mathscr{E}$, the conclusion is true.

    Then for any positive $H$ that $E\left[\int_0^T\int_E H^2 N(dt, de)\right]<\infty$, there exists a sequence of positive increasing simple functions $H_n$ that $E\left[\int_0^T\int_E (H-H_n)^2 N(dt, de)\right]\rightarrow 0$ as $n$ goes to infinity, so
    \begin{equation*}
    \begin{aligned}
    \Delta \left(\int_0^{\cdot}\int_EH\tilde{N}(ds, de)\right)_t&=\lim_{n\rightarrow\infty}\Delta \left(\int_0^{\cdot}H_n\tilde{N}(ds, B)\right)_t=\lim_{n\rightarrow\infty}\int_E H_nN(\{t\}, de)\\&=\int_E HN(\{t\}, de)
    \end{aligned}
    \end{equation*}
    \end{proof}
    \begin{proposition}
    If $H$ is $E$-progressive and $E\left[\int_0^T\int_E H^2 N(dt, de)\right]<\infty$, then we have
    \begin{equation*}
    [H. \tilde{N}, H. \tilde{N}]_t=\int_0^t\int_E H^2N(ds, de)
    \end{equation*}
    \end{proposition}
    \begin{proof}
    The proof is the same as above.
    \end{proof}
\section{Statement of the Problem}
\label{sec:alg}

Given time duration $T>0$, let $\{T_n\}_{n\ge 1}$ be the jump time of $N([0, t]\times E)$,  $T_n:=\inf\left\{t\mid N([0, t]\times E)\ge n\right\}$, then $\{T_n\}_{n\ge 1}$ is a sequence of stopping times that strictly increasing. Let $U$ be a nonempty subset of $R$. We define the admissible control set
\begin{equation}
\begin{aligned}
    U_{ad}\!=\!\Bigg\{u\mid &u \text{ is progressive, taking values in $U$, }\sup_{0\le t\le T}E\left[|u_t|^p\right]\!<\!\infty \text{ for any }p\!>\!1,\\ & \text{and } E\int_0^T|u_t|^2N(dt,E)<\infty\Bigg\}
\end{aligned}
\end{equation}
    For any admissible control $u\in U_{ad}$ and initial state $x_0\in R$, we consider the following progressive stochastic system with jumps:
    \begin{equation}\label{equation}
    X_t=x_0+\int_0^t b(s, X_s, u_s)ds+\int_0^t \sigma(s, X_s, u_s)dB_s+\int_0^t \int_E c(s, X_{s-}, u_s, e)\tilde{N}(ds, de)
    \end{equation}
    along with the cost functional:
    \begin{equation}
    J(u)=E\left[\int_0^T f(t, X_t, u_t)dt+g(X_T)\right]
    \end{equation}
    where $b:\Omega\times [0, T]\times R \times R\rightarrow R$, $\sigma:\Omega\times [0, T]\times R \times R\rightarrow R$, $c:\Omega\times [0, T]\times R \times R\times E\rightarrow R$, $f:\Omega\times [0, T]\times R \times R\rightarrow R$, $g:\Omega\times R\rightarrow R$. The optimal control is to find an element $u\in U_{ad}$ such that
     \[J(u)=\inf_{v\in U_{ad}}J(v).\]
We aim at finding necessary conditions for an optimal control in $U_{ad}$.
    We need the following assumption.\\
    \textbf{Assumption H:}\\
    \begin{itemize}
    \item $b, \sigma, f$ is $\mathscr{G}\otimes \mathscr{B}(R)\otimes \mathscr{B}(R)/\mathscr{B}(R)$ measurable, $c$ is $\mathscr{G}\otimes\mathscr{E}\otimes \mathscr{B}(R)\otimes \mathscr{B}(R)/\mathscr{B}(R)$ measurable, $g$ is $\mathscr{F}_T\otimes \mathscr{B}(R)/\mathscr{B}(R)$ measurable.
    \item $b, \sigma, c$ are twice continuously differentiable about $x$ with bounded first and second order derivatives and there is a constant $C$ such that $|(b,\sigma,c)(t,x,u)|\le C(1+|x|+|u|)$.
    \item $f,g$ are twice continuously differentiable about $x$ with bounded second order derivatives and there is a constant $C$ such that $|f_x(t,x,u)|\le C(1+|x|+|u|)$, $|f(t,x,u)|\le C(1+|x|^2+|u|^2)$  and $|g_x(x)|\le C(1+|x|)$, $|g(x)|\le C(1+|x|^2)$.
    \item $E\int_0^T |b(t, \omega, 0, 0)|^2dt<\infty$, $E\int_0^T |\sigma(t, \omega, 0, 0)|^2dt<\infty$, \\
        $E\int_0^T\int_E |c(t, \omega, e, 0, 0)|^2N(ds, de)<\infty$.
    \end{itemize}
    Under the assumption,  we know that there exists a unique solution of \cref{equation} for any admissible control from \cref{eusde} in appendix.

\section{Variation}
\label{sec:experiments}
Since $U$ is not necessarily convex, we employ spike variations. Suppose $u\in U_{ad}$ is the optimal control, for any $\bar{t}\in [0, T]$, the spike variation of $u$ is defined as follow:
    \begin{equation}
        u^{\epsilon}=\left\{
         \begin{aligned}
         &v,\ \text{if } (s, \omega)\in {\mathcal {O}}:=\rrbracket\bar{t}, \bar{t}+\epsilon\rrbracket\backslash \displaystyle\bigcup_{n=1}^{\infty} \llbracket T_{n}\rrbracket\\
         &u,\ \text{other wise}.
         \end{aligned}
        \right.
    \end{equation}
    where $\llbracket T_{n}\rrbracket:=\left\{(\omega, t)\in \Omega\times[0, T]\mid T_n(\omega)=t\right\}$ is the graph of $T_n$, $v$ is a bounded $\mathscr{F}_{\bar{t}}$ measurable function that takes values in $U$. Since $T_n$ is a stopping time, $\llbracket T_{n}\rrbracket$ is a progressive set. Therefore, the spike variation $u^{\epsilon}$ is progressive and it is easy to show that $u^{\epsilon}$ is in $U_{ad}$.
    \begin{figure}[htbp]\label{fvariation}
        \centering
        \includegraphics{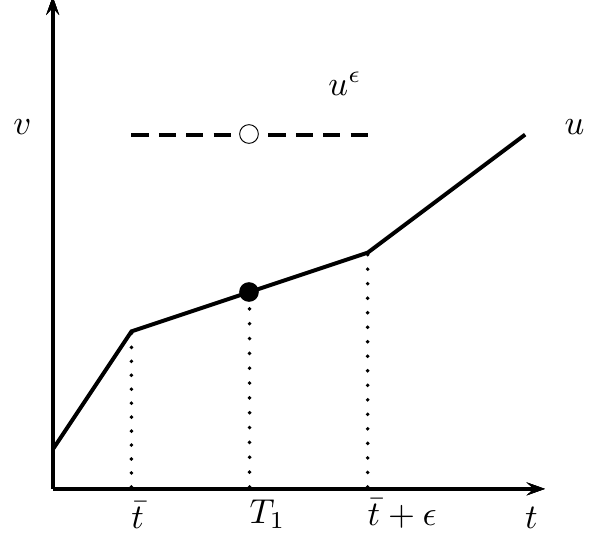}
        \caption{Variation}
    \end{figure}

    The method of variation is showed in \cref{fvariation}. Fix $\omega$, we consider one path of $u^{\epsilon}$ and $u$. The difference between the new method and the traditional method is that if there are jumps in $(t,t+\epsilon]$, for example, as the figure shows that $T_1(\omega)$ is in $(t,t+\epsilon]$, then the value of $u^{\epsilon}$ at $T_1(\omega)$ is equal to $u$ rather than $v$.
    \begin{remark}
    As we know, $T_n$ is not a predictable time, so $\llbracket T_{n}\rrbracket$ is not predictable which means that $u^{\epsilon}$ is not predictable, that's the reason why we need the integrand of the stochastic integral to be progressive. In fact, $T_n$ are totally unpredictable times.
    \end{remark}
    We denote by $X$ the trajectory of $u$, and by $X^{\epsilon}$  the trajectory of $u^{\epsilon}$. By the estimate of SDE and notice that $(Leb\times P)(\llbracket T_{n}\rrbracket)=0$, we can get that:
    \begin{equation*}\small\label{lp estimate for X}
    \begin{aligned}
        &E\left[\sup_{0\le t\le T}\left|X^{\epsilon}_t-X_t\right|^p\right]\le CE\bigg[\left(\int_0^T \left|b(t, X_t, u^{\epsilon}_t)-b(t, X_t, u_t)\right|dt\right)^p\\&+\left(\int_0^T \left|\sigma(t, X_t, u^{\epsilon}_t)-\sigma(t, X_t, u_t)\right|^2dB_t\right)^{p\over 2}\\&+\left(\int_0^T\int_E \left|c(t, X_{t-}, u^{\epsilon}_t, e)-c(t, X_{t-}, u_t, e)\right|^2N(dt, de)\right)^{p\over 2}\bigg]\\ &\le
        CE\bigg[\left(\int_t^{t+\epsilon} \left|u-v\right|dt\right)^p+\left(\int_t^{t+\epsilon} \left|u-v\right|^2dt\right)^{p\over 2}+\left(\int_0^T I_{{\mathcal O}} \left|u-v\right|^2N(dt, E)\right)^{p\over 2}\bigg]\\
    \end{aligned}
    \end{equation*}
    Since there is no jump on ${\mathcal O}$, we have:
    \begin{equation}\label{mest}
    \begin{aligned}
    E\left[\sup_{0\le t\le T}\left|X^{\epsilon}_t-X_t\right|^p\right]=O(\epsilon^p)+O(\epsilon^{p\over 2})
    \end{aligned}
    \end{equation}
    That means the jump term does not influence the order of variation. In fact, if we do not subtract the jump term in variation, $E\left[\left(\int_t^{t+\epsilon} \left|u-v\right|^2N(dt, E)\right)^{p\over 2}\right]$ is always of order $O(\epsilon)$ no matter how large $p$ is. Thanks to this, we can use the method in \cite{peng1990general} to get the desired conclusion. Then we introduce the variation equations:
    \begin{equation}\label{var1}
    \begin{aligned}
    \hat{X}_t&=\int_0^t b_x(s, X_s, u_s)\hat{X}_s+\delta bds+\int_0^t \sigma_x(s, X_s, u_s)\hat{X}_s+\delta\sigma dB_s\\&+\int_0^t \int_E c_x(s, X_{s-}, u_s, e)\hat{X}_{s-}\tilde{N}(ds, de)
    \end{aligned}
    \end{equation}
    and
    \begin{equation}\label{var2}
    \begin{aligned}
    \hat{Y}_t&=\int_0^t b_x(s, X_s, u_s)\hat{Y}_s+\frac{1}{2} b_{xx}(s, X_s, u_s)\hat{X}_s^2ds\\&+\int_0^t \sigma_x(s, X_s, u_s)\hat{Y}_s+ \frac{1}{2} \sigma_{xx}(s, X_s, u_s)\hat{X}_s^2 +\delta\sigma_x\hat{X}_sdB_s\\&+\int_0^t \int_E c_x(s, X_{s-}, u_s, e)\hat{Y}_{s-}+\frac{1}{2} c_{xx}(s, X_{s-}, u_s, e)\hat{X}_{s-}^2\tilde{N}(ds, de)
    \end{aligned}
    \end{equation}
    where $\delta \phi=\phi(s, X_s, u^{\epsilon}_s)-\phi(s, X_s, u_s), \phi=b, \sigma$.
    $\delta\phi_x=\phi_x(s, X_s, u^{\epsilon}_s)-\phi(s, X_s, u_s), \phi=b, \sigma$.

    It is easy to show that \cref{var1,var2} have unique solution. We have some basic estimates about $\hat{X}$ and $\hat{Y}$.
    \begin{lemma}
    For $p\ge 2$, we have the following estimate:
    \begin{equation}
        \left\{
        \begin{aligned}
        &E\left[\sup_{0\le t\le T}|\hat{X}_t|^p\right]\le C\epsilon^{p\over 2}\\
        &E\left[\sup_{0\le t\le T}|\hat{Y}_t|^p\right]\le C\epsilon^p.
        \end{aligned}
        \right.
    \end{equation}
    \end{lemma}
    \begin{proof}
    By the elementary $L^p$ estimate, for $\hat{X}$ we have:
    \begin{equation*}\small
        \begin{aligned}
        &E\left[\sup_{0\le t\le T}|\hat{X}_t|^p\right]\le CE\left[\left(\int_0^T|\delta b|dt\right)^p\right]+CE\left[\left(\int_0^T|\delta \sigma|^2dt\right)^{p\over 2}\right]\\
        &\le CE\left[\left(\int_0^T|u^{\epsilon}_t-u_t|dt\right)^p\right]+CE\left[\left(\int_0^T| u^{\epsilon}_t-u_t|^2dt\right)^{p\over 2}\right]=O(\epsilon^p)+O(\epsilon^{p\over 2})
        \end{aligned}
    \end{equation*}
    for $\hat{Y}$, notice the boundness of $b_{xx}, \sigma_{xx}, c_{xx}$ , we have:
        \begin{align*}\small
        &E\left[\sup_{0\le t\le T}|\hat{Y}_t|^p\right]\le CE\left[\left(\int_0^T| \frac{1}{2}b_{xx}(s, X_s, u_s)\hat{X}^2_s|dt\right)^p\right]\\&+CE\left[\left(\int_0^T| \frac{1}{2} \sigma_{xx}(s, X_s, u_s)\hat{X}^2_s +\delta\sigma_x\hat{X}_s|^2dt\right)^{p\over 2}\right]\\&+CE\left[\left(\int_0^T\int_E|\frac{1}{2} c_{xx}(s, X_{s-}, u_s, e)\hat{X}^2_{s-}|^2N(dt, de)\right)^{p\over 2}\right]\\
        &\le CE\left[\sup_{0\le t\le T}|\hat{X}_t|^{2p}\right]+CE\left[\sup_{0\le t\le T}|\hat{X}_t|^{p}\left(\int_0^T|\delta\sigma_x|^2dt\right)^{p\over 2}\right]\\
        &+CE\left[\sup_{0\le t\le T}|\hat{X}_t|^{2p}N([0, T]\times E)\right]\\
        &=O(\epsilon^p)
        \end{align*}
    \end{proof}
    \begin{lemma}\label{festimate}
    \begin{equation}\small
    \lim_{\epsilon\rightarrow 0}\frac{1}{\epsilon^2}E\left[\sup_{0\le t\le T}|X^{\epsilon}_t-X_t-\hat{X}_t-\hat{Y}_t|^2\right]=0
    \end{equation}
    \end{lemma}
    \begin{proof}
    First find the equation that $X_t+\hat{X}_t+\hat{Y}_t$ satisfies.
    \begin{equation*}\small
    \begin{aligned}
        X_t&+\hat{X}_t+\hat{Y}_t=x_0+\int_0^t  b(s, X_s, u_s)+b_x(s, X_s, u_s)\hat{X}_s+b_x(s, X_s, u_s)\hat{Y}_s\\&+\delta b+\frac{1}{2}b_{xx}(s, X_s, u_s)|\hat{X}_s|^2ds
        +\int_0^t  \sigma(s, X_s, u_s)+\sigma_x(s, X_s, u_s)\hat{X}_s\\&+\sigma_x(s, X_s, u_s)\hat{Y}_s+\delta \sigma+\delta\sigma_x\hat{X}_s+\frac{1}{2}\sigma_{xx}(s, X_s, u_s)|\hat{X}_s|^2dB_s\\
        &+\int_0^t \int_Ec(s, X_{s-}, u_s, e)+c_x(s, X_{s-}, u_s, e)\hat{X}_{s-}+c_x(s, X_{s-}, u_s, e)\hat{Y}_{s-}
        \\&+\frac{1}{2}c_{xx}(s, X_{s-}, u_s, e)\hat{X}^2_{s-}\tilde{N}(ds, de)\\
    \end{aligned}
    \end{equation*}
    Since we have for $\phi=b, \sigma, c$
    \begin{equation*}\small
    \begin{aligned}
    &\phi(s, X_s+\hat{X}_s+\hat{Y}_s, u^{\epsilon}_s, e)-\phi(s, X_s, u_s, e)=\phi(s, X_s+\hat{X}_s+\hat{Y}_s
    , u^{\epsilon}_s, e)-\phi(s, X_s, u^{\epsilon}_s, e)+\delta \phi\\
    &=\delta \phi+\phi_x(s, X_s, u^{\epsilon}_s, e)(\hat{X}_s+\hat{Y}_s)
    +\int_0^1\int_0^1\alpha \phi_{xx}(X_s+\alpha\beta(\hat{X}_s+\hat{Y}_s), u^{\epsilon}_s, e)d\alpha d\beta (\hat{X}_s+\hat{Y}_s)^2\\
    &=\delta \phi+\phi_x(s, X_s, u^{\epsilon}_s, e)(\hat{X}_s+\hat{Y}_s)+A_{\phi} (\hat{X}_s+\hat{Y}_s)^2
    \end{aligned}
    \end{equation*}
    we get
    \begin{equation*}\small
    \begin{aligned}
        X_t+\hat{X}_t+\hat{Y}_t&=x_0+\int_0^t b(s, X_s+\hat{X}_s+\hat{Y}_s, u^{\epsilon}_s)+\Lambda ds\\
        &+\int_0^t \sigma(s, X_s+\hat{X}_s+\hat{Y}_s, u^{\epsilon}_s)+G dB_s\\
        &+\int_0^t\int_E c(s, X_{s-}+\hat{X}_{s-}+\hat{Y}_{s-}, u^{\epsilon}_s, e)+F \tilde{N}(ds, de)
    \end{aligned}
    \end{equation*}
    where
    \begin{equation*}\small
    \begin{aligned}
    &\Lambda=\frac{1}{2}b_{xx}(s, X_s, u_s)|\hat{X}_s|^2- (b_x(s, X_s, u_s^{\epsilon})-b_x(s, X_s, u_s))(\hat{X}_s+\hat{Y}_s)-A_b (\hat{X}_s+\hat{Y}_s)^2\\
    &G=\frac{1}{2}\sigma_{xx}(s, X_s, u_s)|\hat{X}_s|^2- (\sigma_x(s, X_s, u_s^{\epsilon})-\sigma_x(s, X_s, u_s))\hat{Y}_s-A_{\sigma} (\hat{X}_s+\hat{Y}_s)^2\\
    \end{aligned}
    \end{equation*}
    \begin{equation*}\small
    \begin{aligned}
    &F\!=\!\frac{1}{2}c_{xx}(s, X_{s-}, u_s, e)|\hat{X}_{s-}|^2\!- \!(c_x(s, X_{s-}, u_s^{\epsilon},e)\!-\!c_x(s, X_{s-}, u_s,e))\hat{Y}_{s-}-A_{c} (\hat{X}_{s-}+\hat{Y}_{s-})^2\\
    \end{aligned}
    \end{equation*}
    By \cref{festimate}, we have
    \begin{equation*}\small
    \begin{aligned}
    E\left[\left(\int_0^T \Lambda ds\right)^2+\int_0^T G^2 ds+\int_0^T\int_E F^2 N(ds, de)\right]=o(\epsilon^2)
     \end{aligned}
    \end{equation*}
    So by the basic estimate we have
    \begin{equation*}\small
    \begin{aligned}
    E\left[\sup_{0\le t\le T}|X^{\epsilon}_t-X_t-\hat{X}_t-\hat{Y}_t|^2\right]&\le
    CE\left[\left(\int_0^T\Lambda ds\right)^2+\int_0^T G^2 ds\right]\\&+CE\left[\int_0^T\int_E F^2 N(ds, de)\right]
     \end{aligned}
    \end{equation*}
    which shows the result.
    \end{proof}
    Now we get the variation equation for cost functional. We have
    \[J(u)=E\left[\int_0^{T} f(t, X_t, u_t)dt+g(X_T)\right]\]define
    \begin{equation}
    \begin{aligned}
    \hat{J}&=E\left[\int_0^{T} f_x(t, X_t, u_t)(\hat{X}_t+\hat{Y}_t)+\frac{1}{2}f_{xx}(t, X_t, u_t)\hat{X}_t^2+\delta fdt\right]\\&+E\left[g_x(X_T)(\hat{X}_T+\hat{Y}_T)+\frac{1}{2}g_{xx}(X_T)(\hat{X}_T)^2\right]
    \end{aligned}
    \end{equation}
    Then we have the following lemma.
    \begin{lemma}
    \[\lim_{\epsilon\rightarrow 0}\frac{1}{\epsilon}(J(u^{\epsilon})-J(u)-\hat{J})=0\]
    \end{lemma}
    \begin{proof}
    \begin{equation*}\small
    \begin{aligned}
    J(u)+\hat{J}&=E\left[\int_0^{T} f(t, X_t, u_t)+f_x(t, X_t, u_t)(\hat{X}_t+\hat{Y}_t)+\frac{1}{2}f_{xx}(t, X_t, u_t)\hat{X}_t^2+\delta fdt\right]\\&+E\left[g(X_T)+g_x(X_T)(\hat{X}_T+\hat{Y}_T)+\frac{1}{2}g_{xx}(X_T)(\hat{X}_T)^2\right]\\
    &=E\left[\int_0^T  f(t, X_t+\hat{X}_t+\hat{Y}_t, u^{\epsilon}_t)+Hdt\right]+E\left[g(X_T+\hat{X}_T+\hat{Y}_T)+I\right]
    \end{aligned}
    \end{equation*}
    where
    \begin{equation*}\small
    \begin{aligned}
    &H=\frac{1}{2}f_{xx}(s, X_s, u_s)\hat{X}^2_s-\delta f_x(\hat{X}_s+\hat{Y}_s)-A_f (\hat{X}_s+\hat{Y}_s)^2\\
    &I=-\int_0^1\int_0^1\alpha g(X_T+\alpha\beta(\hat{X}_T+\hat{Y}_T))d\alpha d\beta(\hat{X}_T+\hat{Y}_T)^2+\frac{1}{2}g_{xx}(X_T)(\hat{X}_T)^2
    \end{aligned}
    \end{equation*}
    Then
    \begin{equation*}\small
    \begin{aligned}
    |J(u^{\epsilon})-J(u)-\hat{J}|^2&\le CE\left[\int_0^T |f(t, X_t+\hat{X}_t+\hat{Y}_t, u^{\epsilon}_t)-f(t, X^{\epsilon}_t, u^{\epsilon}_t)|^2dt+\left(\int_0^THdt\right)^2\right]\\
    &+E\left[\left|g(X_T+\hat{X}_T+\hat{Y}_T)-g(X^{\epsilon}_T)\right|^2+I^2\right]\\
    &\le CE\left[\sup_{0\le t\le T}|X^{\epsilon}_t-X_t-\hat{X}_t-\hat{Y}_t|^2\right]+E\left[\left(\int_0^THdt\right)^2+I^2\right]\\
    &=o(\epsilon^2)
    \end{aligned}
    \end{equation*}
    By the same method we can show that $E\left[\left(\int_0^THdt\right)^2+I^2\right]=o(\epsilon^2)$,  which proves the result.
    \end{proof}

\section{Adjoint Equations and the Maximum Principle}

We introduce the first order and second order adjoint equation. \\
    First order:
    \begin{equation}\small\label{p}
    \begin{aligned}
    p_t&=g_x(X_T)+\int_t^T \left(b_xp_s+\sigma_xq_s+f_x+\int_E \mathbb{E}[c_x|\mathscr{P}\otimes\mathscr{E}]k_s\lambda(de)\right)ds\\&-\int_t^T q_sdB_s-\int_t^T\int_E k_s\tilde{N}(ds, de)
    \end{aligned}
    \end{equation}
    And the second order:
    \begin{equation}\small\label{P}
    \begin{aligned}
    P_t&=g_{xx}(X_T)+\int_t^T \Big(2b_xP_s+2\sigma_xQ_s+f_{xx}+b_{xx}p_s+\sigma_{xx}q_s+P_s\sigma_x^2\\&+\int_E  \mathbb{E}[(c_x^2+2c_x)|\mathscr{P}\otimes\mathscr{E}]K_s+
    \mathbb{E}[c_{xx}|\mathscr{P}\otimes\mathscr{E}]k_s+\mathbb{E}[c_x^2|\mathscr{P}\otimes\mathscr{E}]P_s\lambda(de)\Big)ds\\&-\int_t^T Q_sdB_s-\int_t^T\int_E K_s\tilde{N}(ds, de)
    \end{aligned}
    \end{equation}
    where $\phi_x=\phi_x(t, X_t, u_t), \phi_{xx}=\phi_{xx}(t, X_t, u_t).$


     In order to get the existence and uniqueness of the two BSDE above, we refer to Lemma 2.4 in \cite{Tang1994necessary}. Since $\phi_x, \phi_{xx}$ are bounded, there exists a unique solution of (\ref{p})  $(p,q,k)\in S^2[0,T]\times M^2[0,T]\times F^2[0,T]$ and a unique solution of (\ref{P})  $(P,Q,K)\in S^2[0,T]\times M^2[0,T]\times F^2[0,T]$, where
      \[S^2[0, T]:=\left\{Y\mid Y \text{ has c\`adl\`ag paths, adapted and }E\left[\sup_{0\le t\le T}|Y_t|^2\right]<\infty\right\}\]with norm $\Vert Y\Vert^2=E\left[\sup_{0\le t\le T}|Y_t|^2\right]$,
      \[M^2[0, T]=\left\{Z\mid Z\text{ is predictable and }E\left[\int_0^T|Z_s|^2ds\right]<\infty\right\}\]
      with norm $\Vert Z\Vert^2=E\left[\int_0^T|Z_s|^2ds\right]$, and
      \[F^2[0, T]=\left\{K\mid K \text{ is $E$-predictable and }E\left[\int_0^T\int_E |K_s|^2\lambda(de)dt<\infty\right]\right\}\]
      with norm $\Vert K\Vert^2=E\left[\int_0^T\int_E|K_t|^2N(dt, de)\right]$.
    Next, we need an It\^o's formula for processes with jumps referring to Theorem 32 and Theorem 33 from \cite{book:12482}.
    \begin{lemma}
    Let $X^1,X^2,...,X^d$ be semimartingales, and $F$ be a $C^2$ function on $R^d$. Set $X=(X^1,X^2,...,X^d)$, then
    \begin{equation*}\small
    \begin{aligned}
    F(X_t)\!-\!F(X_0)&\!=\!\sum^{d}_{i=1}\int_0^t\frac{\partial F}{\partial x_i}(X_{s-})dX^i_s\!+\!\frac{1}{2}\sum^{d}_{i=1,j=1}\int_0^t\frac{\partial^2 F}{\partial x_i\partial x_j}(X_{s-})d[X^i,X^j]_s\!+\!\sum_{s\le t}\eta_s(F)
    \end{aligned}
    \end{equation*}
    where
    \begin{equation*}
    \eta_s(F)=F(X_s)-F(X_{s-})-\sum^d_{i=1}\frac{\partial F}{\partial x_i}(X_{s-})\Delta X^i_s-\frac{1}{2}\sum^{d}_{i=1,j=1}\frac{\partial^2 F}{\partial x_i\partial x_j}(X_{s-})\Delta X_s^i\Delta X_s^j
    \end{equation*}
    and
    \begin{equation*}
     \Delta X_s^i= X^i_s-X^i_{s-}
    \end{equation*}
    \end{lemma}
    Apply It\^o's formula for $p_t\hat{X}_t, p_t\hat{Y}_t$ and $P_t|\hat{X}_t|^2$, we get
    \begin{equation}\small\label{px}
    \begin{aligned}
    E\left[p_T\hat{X}_T\right]&=E\int_0^Tp_{t-}d\hat{X}_t+E\int_0^T\hat{X}_{t-}dp_t+E[p, \hat{X}]_T\\
     &=E\int_0^T\left(p_t\delta b+q_t\delta\sigma-\hat{X}_tf_x\right)dt
    \end{aligned}
    \end{equation}
    and
    \begin{equation}\small\label{py}
    \begin{aligned}
     &E\left[p_T\hat{Y}_T\right]=E\int_0^Tp_{t-}d\hat{Y}_t+E\int_0^T\hat{Y}_{t-}dp_t+E[p, \hat{Y}]_T\\&=E\int_0^T\left(\frac{1}{2}b_{xx}p_t|\hat{X}_t|^2+
     \frac{1}{2}\sigma_{xx}q_t|\hat{X}_t|^2-\hat{Y}_tf_x+\delta\sigma_x\hat{X}_tq_t
     +\int_E\frac{1}{2}\mathbb{E}[c_{xx}|\mathscr{P}\otimes\mathscr{E}]k_t\hat{X}_t^2\right)dt\\
    \end{aligned}
    \end{equation}
    and
    \begin{equation}\small\label{pxx}
    \begin{aligned}
    E\left[P_T|\hat{X}_T|^2\right]&=E\left[\int_0^T|\hat{X}_{t-}|^2dP_t+\int_0^T2P_{t-}\hat{X}_{t-}d\hat{X}_t\right]\\
    &+E\left[\int_0^TP_{t-}d[\hat{X},\hat{X}]_t+\int_0^T2\hat{X}_{t-}d[\hat{X}, P]_t+\sum_{t\le T}\Delta P_t(\Delta\hat{X}_t)^2\right]\\
    &=E\int_0^TP_t(\delta\sigma)^2-\hat{X}_t^2\left(f_{xx}+p_tb_{xx}+q_t\sigma_{xx}+\int_Ek_t\mathbb{E}[c_{xx}
      |\mathscr{P}\otimes\mathscr{E}]\lambda(de)\right)dt\\
    &+E\int_0^T\left(2P_t\hat{X}_t\delta b+2Q_t\hat{X}_t\delta\sigma+2P_t\sigma_x\hat{X}_t\delta\sigma\right) dt
    \end{aligned}
    \end{equation}
    In \ref{pxx}, we use the fact
    \begin{equation*}\small
    \begin{aligned}
    \sum_{t\le T}\Delta P_t(\Delta\hat{X}_t)^2&=\sum_{t\le T}\int_EK_tN(\{t\},de)\left(\int_Ec_x\hat{X}_{t-}N(\{t\},de)\right)^2\\
    &=\sum_{t\le T}\int_E K_tc^2_x\hat{X}^2_{t-}N(\{t\},de)\\
    &=\int_0^T\int_EK_tc^2_x\hat{X}^2_{t-}N(dt,de)
    \end{aligned}
    \end{equation*}
    The second equality follows from the fact that for any $A\in \mathscr{E}$, $N(\{t\},A)=1$ or $0$.

    From \cref{px,py,pxx}, we can get the form of $g_x(X_T)(X_T+Y_T)$ and $g_{xx}(X_T)X_T^2$. Then we have

    \begin{equation}\small\label{Jhat}
    \begin{aligned}
    \hat{J}&=E\left[\int_0^{T} \left(p_t\delta b+q_t\delta\sigma+\delta f+\frac{1}{2}P_t(\delta\sigma)^2\right)dt\right]+o(\epsilon)\\
    \end{aligned}
    \end{equation}
    where $o(\epsilon)$ represents $E\left[\int_0^T\left(\delta\sigma_x\hat{X}_tq_t+P_t\sigma_x\hat{X}_t\delta\sigma+P_t\hat{X}_t\delta b+\hat{X}_t\delta\sigma Q_t\right)dt\right]. $

    We define $H(t, x, u, p, q);=pb(t, x, u)+q\sigma(t, x, u)+f(t, x, u)$.  Then we have

    \begin{theorem}
    Let Assumption H be satisfied. Assume that $u$ is the optimal control, $X$ is the trajectory of $u$, and $(p, q)$ satisfies \cref{p}, and $P$ satisfies \cref{P}. Then , we have a.e a.s: for any $v\in U$,
    \begin{equation}
    \begin{aligned}
    H(t, X_t, v, p_t, q_t)-H(t, X_t, u_t, p_t, q_t)
    +\frac{1}{2}P_t(\sigma(t, X_t, v)-\sigma(t, X_t, u_t))^2\ge 0.
    \end{aligned}
    \end{equation}
    \end{theorem}
    \begin{proof}
    Notice that $\displaystyle\bigcup_{n=1}^{\infty} \llbracket T_{n}\rrbracket$ is negligible under $P\times Leb$,  so by \cref{Jhat} we have
    \begin{equation*}
    \begin{aligned}
    \hat{J}&=E\Bigg[\int_0^T I_{(\bar{t}, \bar{t}+\epsilon]}\Bigg(p_t(b(t, X_t, v)-b(t, X_t, u))+q_t(\sigma(t, X_t, v)-\sigma(t, X_t, u))\\&+
    (f(t, X_t, v)-f(t, X_t, u))+\frac{1}{2}P_t(\sigma(t, X_t, v)-\sigma(t, X_t, u))^2\Bigg)dt\Bigg]+o(\epsilon)
    \end{aligned}
    \end{equation*}
    then both sides are divided by $\epsilon$ and let $\epsilon\rightarrow 0$, we have for a.e $\bar{t}$
    \begin{equation*}
    \begin{aligned}
    E\left[\left(H(\bar{t}, X_{\bar{t}}, v, p_{\bar{t}}, q_{\bar{t}})-H(\bar{t}, X_{\bar{t}}, u, p_{\bar{t}}, q_{\bar{t}})
    +\frac{1}{2}P_{\bar{t}}(\sigma(\bar{t}, X_{\bar{t}}, v)-\sigma(\bar{t}, X_{\bar{t}}, u))^2\right)\right]\ge 0
    \end{aligned}
    \end{equation*}
    Then for any $A\in \mathscr{F}_{\bar{t}}$ and $w\in U$, let $v=wI_A+uI_{A^c}$, we have
    \begin{equation*}
    \begin{aligned}
    E\Bigg[I_A&\Bigg(H(\bar{t}, X_{\bar{t}}, w, p_{\bar{t}}, q_{\bar{t}})\!-\!H(\bar{t}, X_{\bar{t}}, u, p_{\bar{t}}, q_{\bar{t}})
    \!+\!\frac{1}{2}P_{\bar{t}}(\sigma(\bar{t}, X_{\bar{t}}, w)\!-\!\sigma(\bar{t}, X_{\bar{t}}, u))^2\Bigg)\Bigg]\!\ge\! 0
    \end{aligned}
    \end{equation*}
    which means a.e a.s
    \begin{equation*}
    \begin{aligned}
    H(\bar{t}, X_{\bar{t}}, w, p_{\bar{t}}, q_{\bar{t}})-H(\bar{t}, X_{\bar{t}}, u, p_{\bar{t}}, q_{\bar{t}})
    +\frac{1}{2}P_{\bar{t}}(\sigma(\bar{t}, X_{\bar{t}}, w)-\sigma(\bar{t}, X_{\bar{t}}, u))^2\ge 0
    \end{aligned}
    \end{equation*}
    \end{proof}

\section{Conclusions}
\label{sec:conclusions}
    In this paper, we introduce a new method of variation.  With the help of our new variation, we overcome the difficulty that the jumps caused in $L^p$ estimate, in other words, \cref{mest} holds, and the order of this estimate grows with the growth of $p$, this feature is important to make the variation equations effective.

    The form of our maximum principle with jumps is the same as the form of maximum principle in \cite{peng1990general} without jumps. The reason is that both maximum principles are hold a.e a.s. In our case with jumps, since the measure of all jumps' graphs is a negligible set under $P\times Leb$, jumps does not influence our result. In other words, our maximum principle only describe the optimal control on the area that $N$ is continuous, it has no information about the optimal control on the time $N$ jumps. However, this is a rigorous maximum principle obtained in a clear and concise mathematical framework and laid a solid foundation for further related theoretical and application research. Our future research is to find a way to characterize optimal control on the time $N$ jumps and explore wide applications in practice.

\appendix
\section{Existence and Uniqueness of SDE and $L^p$ estimate}
Given a SDE with jump:
    \begin{equation}\label{SDE}
    X_t=x_0+\int_0^t b(s, X_s)ds+\int_0^t \sigma(s, X_s)dB_s+\int_0^t \int_E c(s, X_{s-}, e)\tilde{N}(ds, de)
    \end{equation}
    where $x_0\in R^n$, $b:\Omega\times [0, T]\times R^n \rightarrow R^n$, $\sigma:\Omega\times [0, T]\times R^n \rightarrow R^{n\times d}$, $c:\Omega\times [0, T]\times R^n\times E\rightarrow R^n$, $d$ is the dimension of Brownian Motion and $n$ is the dimension of $X$. We introduce a Banach space \[S^2[0, T]:=\left\{X\mid X \text{ has c\`adl\`ag paths and adapted and }E\left[\sup_{0\le t\le T}|X_t|^2\right]<\infty\right\}\]with norm $\Vert X\Vert^2=E\left[\sup_{0\le t\le T}|X_t|^2\right]$. We have the following assumptions:\\
    \textbf{Assumption H1:}\\
    \begin{itemize}
    \item $b$ is $\mathscr{G}\otimes \mathscr{B}(R^n)/\mathscr{B}(R^n)$ measurable, $\sigma$ is $\mathscr{G}\otimes \mathscr{B}(R^n) /\mathscr{B}(R^{n\times d})$ measurable, $c$ is $\mathscr{G}\otimes\mathscr{E}\otimes \mathscr{B}(R^n) /\mathscr{B}(R^n)$ measurable.
    \item $b, \sigma, c$ are uniform lipschitz continuous about $x$.
    \item $E\int_0^T |b(t, \omega, 0)|^2dt<\infty$, $E\int_0^T |\sigma(t, \omega, 0)|^2dt<\infty$, \\ $E\int_0^T\int_E |c(t, \omega, 0,e)|^2N(ds, de)<\infty$.
    \end{itemize}
    \begin{theorem}\label{eusde}
    Under Assumption H1,  \cref{SDE} has a unique solution in $S^2[0, T]$.
    \end{theorem}
    \begin{proof}
    First we show that for each $X$ in $S^2[0, T]$, $\int_0^t\int_E c(s, X_{s-}, e)\tilde{N}(ds, de)$ is well defined.
    Since $X_{s-}$ is left continuous, it is progressive, and $c(s,\omega,x,e)$ is $E$-progressive by assumption. This implies that $c(s, X_{s-}, e)$ is $E$-progressive. And for any $t\in[0, T]$
    \begin{equation*}
    \begin{aligned}
    E\left[\int_0^t\int_E |c(s, X_{s-}, e)|^2N(ds, de)\right]&\le CE\left[\int_0^t\int_E |c(s, \omega,0,e)|^2+|X_{s-}|^2N(ds, de)\right]\\&\le C+Ct\lambda(E)E\left[\sup_{0\le s\le t}|X_s|^2\right]<\infty
    \end{aligned}
    \end{equation*}
    That means that the stochastic integral is well defined.

    Next we show that there is a unique solution in small time duration. We construct a map from $S^2[0, T]$ to $S^2[0, T]$:
    \begin{equation*}
    \mathscr{T}(X)_t=x_0+\int_0^t b(s, X_s)ds+\int_0^t\sigma(s, X_s)dB_s+\int_0^t\int_E c(s, X_{s-}, e)\tilde{N}(ds, de)
    \end{equation*}
    It is easy to show that the image of $\mathscr{T}$ is actually in $S^2[0, T]$, then we show it is a contraction. For any $X, Y\in S^2[0, T]$,
    \begin{equation*}\small
    \begin{aligned}
    \Vert\mathscr{T}(X)-\mathscr{T}(Y)\Vert^2&\le CE\left[\left(\int_0^T |b(t, X_t)-b(t, Y_t)|dt\right)^2\right]\\&+CE\left[\sup_{0\le t\le T}\left|\int_0^t \sigma(t, X_t)-\sigma(t, Y_t)dB_t\right|^2\right]\\&+CE\left[\sup_{0\le t\le T}\left|\int_0^t\int_E c(t, X_{t-}, e)-c(t, Y_{t-}, e)\tilde{N}(dt, de)\right|^2\right]\\
    &\le C\Vert(X-Y)\Vert^2(T+T^2)+CE\left[\int_0^T\int_E|X_{t-}-Y_{t-}|^2\lambda(de)dt\right]\\
    &\le C\Vert(X-Y)\Vert^2(T+T^2)
    \end{aligned}
    \end{equation*}
    $C$ is a constant not related to $T$ but changed every step. So we can choose $T$ small enough that $C(T+T^2)<1$, then $\mathscr{T}$ is a contraction.

    For arbitrary $T$,  we can split $T$ into finite small pieces, then we get a unique solution on each piece and connect them together.
    \end{proof}
    \begin{remark}
    The difference between our results and the results in \cite{rong2006theory} is that in our case $c$ is $E$-progressive and in \cite{rong2006theory}'s case $c$ is $E$-predictable. In fact from the proof above, the difference is slight.
    \end{remark}

    The theorem below is the $L^p$ estimate:
    \begin{theorem}\label{lp est}
    For $p\ge 2$, suppose that $X^i, i=1, 2. $ is the solution of the follow equations
    \begin{equation*}\small
    \begin{aligned}
    X^i_t=x^i_0+\int_0^t b^i(s, X^i_s)ds+\int_0^t \sigma^i(s, X^i_s)dB_s+\int_0^t \int_E c^i(s, X^i_{s-}, e)\tilde{N}(ds, de)
    \end{aligned}
    \end{equation*}
    which satisfy assumption H1, then we have
    \begin{equation}\small
    \begin{aligned}
    &E\left[\sup_{0\le t\le T}|X^1_t-X_t^2|^p\right]\le C|x_0^1-x_0^2|^p+CE\left[\left(\int_0^T|b^1(t, X^1_t)
    -b^2(t, X^1_t)|dt\right)^p\right]\\&+CE\left[\left(\int_0^T|\sigma^1(t, X^1_t)
    -\sigma^2(t, X^1_t)|^2dt\right)^\frac{p}{2}\right]\\&+CE\left[\left(\int_0^T\int_E|c^1(t, X^1_{t-}, e)
    -c^2(t, X^1_{t-}, e)|^2N(dt, de)\right)^\frac{p}{2}\right]
    \end{aligned}
    \end{equation}
    $C$ is a positive real number related to $p, T$ and the Lipschitz constant.
    \end{theorem}
    \begin{proof}
    By a simple calculation, we have
    \begin{equation}\small\label{astep}
    \begin{aligned}
    &E\left[\sup_{0\le t\le T}|X^1_t-X_t^2|^p\right]\le C|x_0^1-x_0^2|^p+CT^pE\left[\sup_{0\le t\le T}|X^1_t-X_t^2|^p\right]\\&+CT^{p\over 2}E\left[\sup_{0\le t\le T}|X^1_t-X_t^2|^p\right]+CE\left[\left(\int_0^T|X^1_{t-}-X_{t-}^2|^2N(dt, E)\right)^{p\over 2}\right]\\&+CE\left[\left(\int_0^T|b^1(t, X^1_t)
    -b^2(t, X^1_t)|dt\right)^p\right]+CE\left[\left(\int_0^T|\sigma^1(t, X^1_t)
    -\sigma^2(t, X^1_t)|^2dt\right)^\frac{p}{2}\right]\\&+CE\left[\left(\int_0^T\int_E|c^1(t, X^1_{t-}, e)
    -c^2(t, X^1_{t-}, e)|^2N(dt, de)\right)^\frac{p}{2}\right].
    \end{aligned}
    \end{equation}
    Now we set $H_t:=|X^1_{t-}-X_{t-}^2|^2, A_t:=\int_0^tH_sN(ds, E)$, then $A_t$ is a pure jump process and so is $A_t^{p\over 2}$. Notice that the jump time of $A^{p\over 2}$ is also a jump time of $N$ and the jump size of $N$ is always equal to $1$. So we have
    \begin{equation*}\small
    \begin{aligned}
    A_T^{p\over 2}&=\sum_{s\le T}A_s^{p\over 2}-A_{s-}^{p\over 2}=\sum_{s\le T}\left(A_s^{p\over 2}-A_{s-}^{p\over 2}\right)I_{\{N(\{s\}, E)\neq 0\}}\\
    &=\sum_{s\le T}\left(|A_{s-}+H_s|^{p\over 2}-A_{s-}^{p\over 2}\right)N(\{s\}, E)\\
    &=\int_0^T|A_{s-}+H_s|^{p\over 2}-A_{s-}^{p\over 2}N(ds, E)\\
    &\le C\int_0^TA_{s-}^{p\over 2}+H_s^{p\over 2}N(ds, E).\\
    \end{aligned}
    \end{equation*}
    Since $A_{\cdot-}$ and $H$ are predictable, we have
    \begin{equation*}\small
    \begin{aligned}
    E\left[A_T^{p\over 2}\right]\le CE\left[\int_0^TA_s^{p\over 2}+H_s^{p\over 2}ds\right]
    \le CTE\left[A_T^{p\over 2}\right]+CTE\left[\sup_{0\le t\le T}|X^1_t-X_t^2|^p\right].
    \end{aligned}
    \end{equation*}
    So if we choose $T$ small enough that $CT<1$, then we have
    \[E\left[A_T^{p\over 2}\right]\le \frac{CT}{1-CT}E\left[\sup_{0\le t\le T}|X^1_t-X_t^2|^p\right].\]
    By the calculation of \cref{astep}, choose $T$ smaller if necessary, we have the estimate in small time duration by subtract $\left(T^p+T^{p\over 2}+\frac{CT}{1-CT}\right)E\left[\sup_{0\le t\le T}|X^1_t-X_t^2|^p\right]$ on both sides of \cref{astep}. For any $T$, we can split $T$ into small pieces and get the desired conclusion.
    \end{proof}
    \begin{remark}
    Without loss of generality, we  assume
    \begin{equation}\small\label{tempa}
    \begin{aligned}
    E\left[\sup_{0\le t\le T}|X^1_t-X_t^2|^p\right]<\infty
    \end{aligned}
    \end{equation}
    in the preceding proof. If not, we can introduce a sequence of stopping times that make \cref{tempa} true, then get the $L^p$ estimate with stopping time and take limits. So we can subtract that term on both sides of \cref{astep}.
    \end{remark}

     \bibliographystyle{siamplain}
    \bibliography{wenxian}
\end{document}